\documentclass{article}

\usepackage{amsmath,amsthm,amssymb,enumerate,epsfig,color,graphicx,epstopdf}

\usepackage{array}
\newcolumntype{C}[1]{>{\centering\arraybackslash$}p{#1}<{$}}

\def\<{\langle}
\def\>{\rangle}

\newtheorem{theorem}{Theorem}[section]
\newtheorem{lemma}[theorem]{Lemma}

\newtheorem{corollary}[theorem]{Corollary}

\newtheorem{example}[theorem]{Example}

\title{Geometric embeddings of braid groups do not merge conjugacy classes}
\author{Juan Gonz\'alez-Meneses\footnote{Partially supported by the Spanish research projects MTM2010-19355, P09-FQM-5112, FEDER and the Australian Research Council's Discovery Project DP1094072.}}
\date{July 16, 2013}

\begin{document}

\maketitle


\begin{abstract}
An embedding of the $m$-times punctured disc into the $n$-times punctured disc, for $n>m$, yields an embedding of the braid group on $m$ strands $B_m$ into the braid group on $n$ strands $B_n$, called a geometric embedding. The main example consists of adding $n-m$ trivial strands to the right of each braid on $m$ strands. We show that geometric embeddings do not merge conjugacy classes, meaning that if the images of two elements in $B_m$ by a geometric embedding are conjugate in $B_n$, the original elements are conjugate in $B_m$. We also show that the result does not hold, in general, for geometric embeddings of mapping class groups.
\end{abstract}

\section{Introduction}

Braid groups on $n$ strands $B_n$ were introduced by Artin~\cite{Artin1} and can be defined in several different ways. The easiest one is by means of the following presentation, also due to Artin~\cite{Artin}
$$
   B_n = \left\langle \sigma_1,\ldots,\sigma_{n-1}  \left| \begin{array}{ll} \sigma_i\sigma_j=\sigma_j\sigma_i & |i-j|>1 \\
   \sigma_i\sigma_j\sigma_i = \sigma_j \sigma_i \sigma_j & |i-j|=1 \end{array}  \right. \right\rangle
$$
If we see braids as homotopy classes of disjoint strands relative to their endpoints~\cite{Birman}, the generator $\sigma_i$ corresponds to the crossing of strands in positions $i$ and $i+1$, in the positive sense (see Figure~\ref{F:generators}).

\begin{figure}
\begin{center}
  \includegraphics{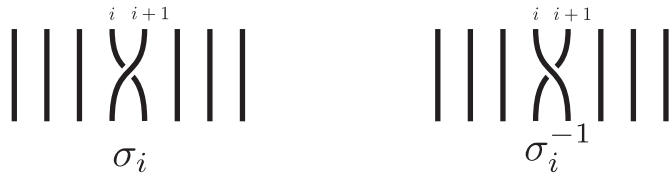}
\end{center}
\caption{Classical (Artin) generators of $B_n$}
\label{F:generators}
\end{figure}

The group $B_n$ can also be seen as the mapping class group of the $n$-times punctured disc, relative to the boundary. More precisely, if $\mathbb D_n$ is a topological disc with $n$ punctures ($n$ interior points removed), we can consider the set $Homeo^+(\mathbb D_n, \partial \mathbb D_n)$ of orientation preserving homeomorphisms from $\mathbb D_n$ to itself fixing the boundary pointwise. Endowing $Homeo^+(\mathbb D_n, \partial \mathbb D_n)$ with the usual compact-open topology, one defines the mapping class group $MCG(\mathbb D_n, \partial \mathbb D_n) = \pi_0(Homeo^+(\mathbb D_n, \partial \mathbb D_n))$. That is, $MCG(\mathbb D_n, \partial \mathbb D_n)$ is the set of orientation preserving homeomorphisms of $\mathbb D_n$ (fixing the boundary pointwise), up to isotopy. It turns out that $B_n\simeq MCG(\mathbb D_n,\partial \mathbb D_n)$~\cite{Birman}.

There is a natural way to embed $B_m$ into $B_n$ for $m<n$: Choose a topological embedding  $\varphi:\: \mathbb D_m \to \mathbb D_n$ which sends punctures to punctures and $\partial \mathbb D_m$ to the interior of $\mathbb D_n$, and then extend every homeomorphism $f$ of $\mathbb D_m$ fixing the boundary pointwise to a homeomorphism $\varphi(f)$ of $\mathbb D_n$ which equals $\varphi\circ f \circ \varphi^{-1}$ in $\varphi(\mathbb D_m)$ and is the identity outside $\varphi(\mathbb D_m)$. This map which sends $f$ to $\varphi(f)$ is injective and respects isotopies, so it induces an embedding $\varphi:\: B_m \to B_n$, which is called a {\it geometric embedding} (see~\cite{Wajnryb}).  Notice that we denoted three different maps (applied to points, to homeomorphisms and to mapping classes) with the same letter $\varphi$, hoping that this will not cause confusion.

The simplest example of a geometric embedding occurs when the disc embedding $\varphi$ is given by the inclusion, in the complex plane $\mathbb C$, of the $m$-times punctured disc of diameter $m-1/2$ whose punctures are the integers $1,\ldots,m$, into the $n$-times punctured disc of diameter $n$ whose punctures are the integers $1,\ldots,n$ (see Figure~\ref{F:eta}). We call this the {\it standard embedding} of $\mathbb D_m$ into $\mathbb D_n$, and will denote it by $\eta$. One can easily determine the {\it standard geometric embedding} $\eta:\: B_m\to B_n$ in terms of Artin's presentation, as in this case $\eta(\sigma_i)=\sigma_i$ for $i=1,\ldots,m-1$.

\begin{figure}
\begin{center}
  \includegraphics{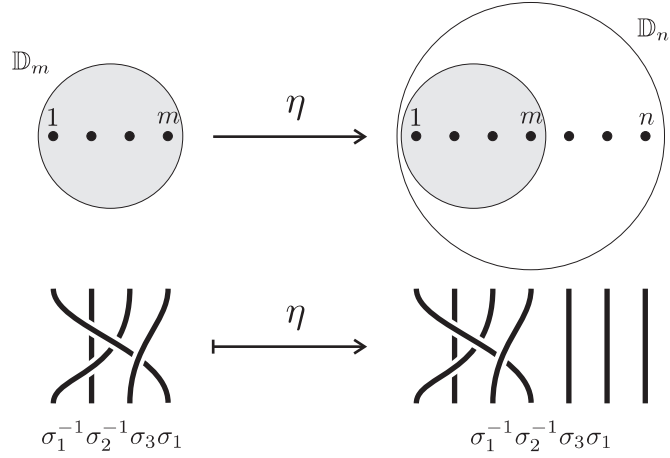}
\end{center}
\caption{The standard geometric embedding of $B_m$ into $B_n$}
\label{F:eta}
\end{figure}

Actually, every geometric embedding $\varphi$ is conjugate to $\eta$, in the following sense. First, $\eta \circ \varphi^{-1}$ sends $\varphi(\mathbb D_m)$ to $\eta(\mathbb D_m)$. Then, as $\mathbb D_n\backslash int(\varphi(\mathbb D_m))$ and $\mathbb D_n\backslash int(\eta(\mathbb D_m))$ are both annuli with $n-m$ punctures, there is a homeomorphism from the former to the latter which coincides with $\eta\circ \varphi^{-1}$ on $\partial(\varphi(\mathbb D_m))$. Gluing both homeomorphisms, we obtain an orientation preserving homeomorphism $g$ of $\mathbb D_n$ (fixing the boundary pointwise) such that $g \circ \varphi = \eta$, the standard embedding of $\mathbb D_m$ into $\mathbb D_n$. Considering $g$ as an element of $B_n$ we obtain
$$
   \begin{array}{rcl}  \eta:\: B_m & \longrightarrow & B_n \\
                        f  & \longmapsto & g \circ \varphi(f) \circ g^{-1}
   \end{array}
$$
as this is precisely the geometric embedding associated to the standard disc embedding $g\circ \varphi = \eta:\: \mathbb D_m \to \mathbb D_n$.

Now the following question arises: Do geometric embeddings merge conjugacy classes of $B_m$? In other words, given a geometric embedding $\varphi:\: B_m\to B_n$, and given two non-conjugate elements $a,b\in B_m$, can $\varphi(a)$ and $\varphi(b)$ be conjugate in $B_n$?

In this paper, we answer the above question in the negative, by showing the following.

\begin{theorem}\label{T:main}
Let $\eta:\: B_m \to B_n$ be the standard geometric embedding. For every $a,b\in B_m$, if $\eta(a)$ and $\eta(b)$ are conjugate in $B_n$, then $a$ and $b$ are conjugate in $B_m$.
\end{theorem}

From the above result the following follows easily:

\begin{corollary}
Let $\varphi:\: B_m \to B_n$ be a geometric embedding. For every $a,b\in B_m$, if $\varphi(a)$ and $\varphi(b)$ are conjugate in $B_n$, then $a$ and $b$ are conjugate in $B_m$.
\end{corollary}

\begin{proof}
Recall that $\eta$ is the geometric embedding corresponding to the disc embedding $g\circ \varphi$ for some homeomorphism $g$ of $\mathbb D_n$. If there exists $\alpha\in B_n$ such that $\varphi(b)=\alpha \varphi(a)\alpha^{-1}$, then $\eta(b)= g\varphi(b)g^{-1} = \left(g \alpha g^{-1}\right)\left( g \varphi(a)g^{-1}\right)\left(g \alpha^{-1} g^{-1}\right) = \left(g \alpha g^{-1}\right)\eta(a)\left(g \alpha g^{-1}\right)^{-1}$. Thus $\eta(a)$ and $\eta(b)$ are conjugate, hence so are $a$ and $b$ by Theorem~\ref{T:main}.
\end{proof}

In order to show Theorem~\ref{T:main}, in Section 2 we will recall some basic facts from Nielsen-Thurston theory of braids and mapping classes, which will be used in Section 3 to prove the result.

As several arguments in the proof of Theorem~\ref{T:main} use the general theory of mapping class groups, one is tempted to try to generalize the result to geometric embeddings of other mapping class groups. That is, if $M$ is a (possibly punctured) surface and $N$ is a subsurface such that $\partial N\cap \partial M=\emptyset$, the injection $i:\: N\hookrightarrow M$ induces a map $i_*:\: MCG(N,\partial N) \rightarrow MCG(M,\partial M)$ which in most cases is injective~\cite{PR}. Here is an example of such a geometric embedding which merges conjugacy classes.

\begin{example}
Let $M_{2,1}$ be the orientable surface of genus 2 with 1 boundary component, and let $M_{3,0}$ be the orientable closed surface of genus 3. Consider the geometric embedding $i:\: M_{2,1} \hookrightarrow M_{3,0}$ represented in Figure~\ref{F:surfaceEmbedding}. The induced map $i_*:\: MCG(M_{2,1},\partial M_{2,1}) \rightarrow MCG(M_{3,0})$ is injective as $\overline{M_{3,0}\backslash M_{2,1}}$ is neither a disc nor a punctured disc~\cite{PR}.

Let $\gamma$ and $\delta$ be two simple closed curves in $M_{2,0}$ as in Figure~\ref{F:surfaceEmbedding}, and let $\tau_{\gamma}$ and $\tau_{\delta}$ be their corresponding Dehn twists.

First we see that $\tau_\gamma$ is not conjugate to $\tau_\delta$ in $MCG(M_{2,1}, \partial M_{2,1})$, as follows: Collapsing the boundary of $M_{2,1}$ to a puncture $p$ yields a group homomorphism from $MCG(M_{2,1}, \partial M_{2,1})$ to $MCG(M_{2,0}\backslash\{p\})$. The image of $\tau_\delta$ is trivial, while the image of $\tau_\gamma$ is not, hence these elements cannot be conjugate.

Now notice that the mapping class $\rho\in MCG(M_{3,0})$ represented by a rotation of angle $2\pi/3$ sends the curve $i(\gamma)$ to the curve $i(\delta)$. This means that, in $MCG(M_{3,0})$, one has $\rho \circ i_*(\tau_{\gamma})\circ  \rho^{-1} = i_*(\tau_\delta)$. Hence $i_*$ merges conjugacy classes.
\end{example}

\begin{figure}
\begin{center}
  \includegraphics{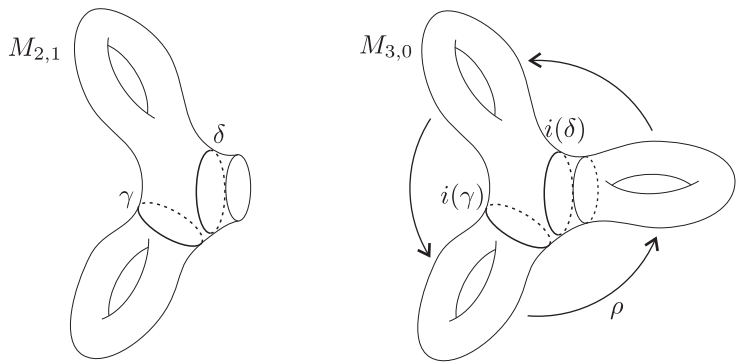}
\end{center}
\caption{A geometric embedding which merges conjugacy classes}
\label{F:surfaceEmbedding}
\end{figure}

\section{Some facts from Nielsen-Thurston theory}

Recall that the braid group $B_n$ can be defined as $B_n=MCG(\mathbb D_n,\partial \mathbb D_n)$. There is a special braid, denoted $\Delta^2$, which corresponds to a Dehn twist along a curve which is parallel to the boundary of $\mathbb D_n$. This braid is called a {\it full twist}, or the {\it Garside element} of $B_n$, and it generates the center of $B_n$~\cite{Chow}.

If we quotient $B_n$ by its center, that is if we consider $B_n/\langle \Delta^2\rangle$, geometrically this corresponds to removing the condition that the boundary of $\mathbb D_n$ is fixed pointwise, so $B_n/\langle \Delta^2\rangle = MCG(\mathbb D_n)$. Notice that we could collapse the boundary of $\mathbb D_n$ to a new puncture, so this group becomes a subgroup of the mapping class group of the $(n+1)$-times punctured sphere.

This allows to study braids using the Nielsen-Thurston theory of mapping classes~\cite{Thurston,FLP}, so braids can be classified into three  geometric types:
\begin{itemize}

 \item {\bf Periodic:} If they have finite order modulo $\Delta^2$.

 \item {\bf Reducible, not periodic:} If they are not periodic, and they preserve a family of disjoint, non-parallel simple closed curves in $\mathbb D_n$, each one enclosing more than 1 and less than $n$ punctures (such curves are called {\it non-degenerate}).

 \item {\bf pseudo-Anosov:} If the induced mapping class of the $(n+1)$-times punctured sphere admits a representative which preserves two transverse measured foliations, scaling the measure of one of them by some real number $\lambda>0$, and the measure of the second one by $\lambda^{-1}$.

\end{itemize}

One can interpret this classification as follows: Periodic braids have a power which preserves {\it all} non-degenerate curves. Reducible (non-periodic) braids have a power which preserves {\it some} non-degenerate curves. Pseudo-Anosov braids have no power preserving a non-degenerate curve.

The most important case in this paper will be the second one: Reducible (but not periodic) braids. By definition, these braids preserve a family of {\it non-degenerate} curves. By~\cite{BLM}, there is a distinguished family of curves associated to a reducible, non periodic braid $\beta$, called the {\it canonical reduction system} of $\beta$ and denoted $CRS(\beta)$.

The canonical reduction system $CRS(\beta)$ of a given braid $\beta$ (not necessarily reducible) is defined as follows. A simple closed curve $\mathcal C$ is said to be {\it essential} for $\beta$ if $\beta^N(\mathcal C)=\mathcal C$ for some $N>0$, and if $\mathcal C'$ is not fixed by any power of $\beta$ for every $\mathcal C'$ having positive geometric intersection with $\mathcal C$ (that is, $\mathcal C'$ cannot be isotoped to be disjoint from $\mathcal C$). Then $CRS(\beta)$ is the set of (isotopy classes of) essential curves for $\beta$.

Some important facts concerning canonical reduction systems are the following~\cite{BLM}:
\begin{itemize}

\item $CRS(\beta)\neq \emptyset$ if and only if $\beta$ is reducible, not periodic.

\item $CRS(\alpha\circ \beta \circ \alpha^{-1}) = \alpha(CRS(\beta))$.

\item $\mathbb D_n \backslash CRS(\beta)$ has several connected components. Given a power $\beta^M$ of $\beta$ which preserves each of these connected components (the permutation induced by $\beta^M$ on the set of connected components is trivial), the restriction of $\beta^M$ to each connected component is either periodic or pseudo-Anosov.

\item $CRS(\beta)$ is the minimal family, under inclusion, of (isotopy classes of) non-degenerate curves satisfying the above property. In other words, a family of curves which satisfies the above property is $CRS(\beta)$ if and only if the family obtained by removing any of its curves does not satisfy the property.

\end{itemize}

One particular property of braids, compared to other mapping classes, is that cutting a punctured disc along disjoint non-degenerate curves yields again punctured discs (each curve plays the role of boundary of the outermost connected component it encloses). This implies that each reducible (non-periodic) braid $\beta\in B_n$ can be decomposed into simpler braids on fewer strands, cutting the disc along $CRS(\beta)$. Usually, in the theory of mapping classes, one takes a power $\beta^M$ such that each connected component is sent to itself, but in the case of braids this is not necessary, as one can see in~\cite{GM}.

In any case, if $\beta\in B_n$ is reducible and non-periodic, the outermost connected component of $\mathbb D_n \backslash CRS(\beta)$ is always preserved by $\beta$, and the restriction of $\beta$ to this component is a braid that we will call $\beta_{ext}$, the external component of $\beta$.

We point out that $\beta_{ext}$ is well defined up to conjugacy, as it depends on the way of collapsing the {\it holes} of the surface to punctures. Although there is a way to properly define the external and internal components of $\beta$ (see~\cite{GM}), we shall not need it, as the only property of $\beta_{ext}$ we care about in this paper is whether it is trivial or not, and this is well defined up to conjugacy.

We will use the above results in the next section to show Theorem~\ref{T:main}.

\section{The standard geometric embedding does not merge conjugacy classes}

 In this section we will show Theorem~\ref{T:main}. Let $\eta: B_m\to B_n$ be the standard geometric embedding, and let $a,b\in B_m$ such that $\eta(a)$ and $\eta(b)$ are conjugate.

If $a=1$ the result follows easily: one has $\eta(a)=1$ which implies that $\eta(b)=1$, and as $\eta$ is injective one obtains that $b=1$ too. We can then assume that both $a$ and $b$ are nontrivial.

First we must point out that $\eta(a)$ preserves the boundary of $\eta(\mathbb D_m)$, which is a non-degenerate curve of $\mathbb D_n$ hence $\eta(a)$ is not pseudo-Anosov. We shall now see that $\eta(a)$ cannot be periodic.

If $\eta(a)$ were periodic, there would be some integers $M,N$ with $N>0$ such that $\eta(a)^N=\Delta^{2M}$. As $\Delta$ is a braid in which every strand crosses every other strand once in the positive sense, $\Delta^{2M}$ is a braid in which every strand crosses every other strand, unless $M=0$. But in our case the $n$-th strand of $\eta(a)^N$ would not cross any other, so we deduce that $M=0$, hence $\eta(a)^N=1$. As braid groups have no torsion~\cite[Theorem 8]{FaN} (see also~\cite{FoN,Dyer,GM_survey}), it would follow that $\eta(a)=1$, so $a=1$, a contradiction. Therefore $\eta(a)$ is not periodic.

From the above arguments we obtain that $\eta(a)$ is reducible, non periodic. Hence $CRS(\eta(a))\neq \emptyset$. We shall now compare $CRS(\eta(a))$ and $CRS(a)$, the latter being possibly empty, as $a$ can be either periodic, reducible or pseudo-Anosov.

\begin{figure}
\begin{center}
  \includegraphics{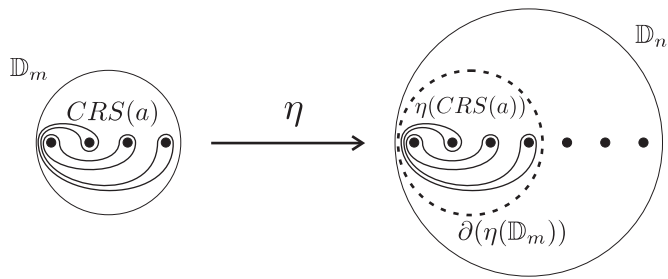}
\end{center}
\caption{The canonical reduction system of $\eta(a)$ is obtained by applying $\eta$ to $CRS(a)$, and possibly adding $\partial(\eta(\mathbb D_m))$.}
\label{F:CRS}
\end{figure}

\begin{lemma}\label{L:CRS}
Under the above conditions, if $a_{ext}$ is trivial, then $CRS(\eta(a))=\eta(CRS(a))$. Otherwise $CRS(\eta(a))=\eta(CRS(a))\cup \partial(\eta(\mathbb D_m))$. Moreover, in either case $\eta(a)_{ext}$ is trivial.
\end{lemma}

\begin{proof}
Recall that $\eta(a)$ is obtained by extending the homeomorphism $a$ of $\mathbb D_m$ by the identity to the whole $\mathbb D_n$, where $\mathbb D_m$ is embedded into $\mathbb D_n$ by the standard embedding $\eta$. This implies in particular that $\partial(\eta(\mathbb D_m))$ is preserved by $\eta(a)$.

Since $CRS(a)$ is preserved by $a$, if we apply the standard disc embedding $\eta$ to this (possibly empty) family of curves, we obtain $\eta(CRS(a))$ which is preserved by $\eta(a)$. Hence $\Gamma = \eta(CRS(a)) \cup \partial(\eta(\mathbb D_m))$ is a family of disjoint, simple, non-parallel, non-degenerate curves preserved by $\eta(a)$. Moreover, we claim that if $\eta(a)^M$ preserves the connected components of $\mathbb D_n\backslash \Gamma$, then $\eta(a)^M$ restricted to each connected component is either periodic or pseudo-Anosov. Indeed, $\eta(a)^M$ restricted to the outermost component of $\mathbb D_n\backslash \Gamma$ is the identity, while $\eta(a)^M$ restricted to any other component is equal to $a^M$ restricted to a component of $\mathbb D_m\backslash CRS(a)$, hence it is either periodic or pseudo-Anosov.

Therefore $CRS(\eta(a))$ is a subset of $\Gamma$, and we only need to see which curves of $\Gamma$ can be removed so that the restriction of $\eta(a)^M$ to each connected component is still either periodic or pseudo-Anosov.

Given $\mathcal C \in \eta(CRS(a))$, removing $\mathcal C$ from $\Gamma$ merges two connected components. That is, if we consider $\Gamma'=\Gamma\backslash \{\mathcal C\}$, there is a connected component of $\mathbb D_n \backslash \Gamma'$ which is the union of $\mathcal C$ and of two connected components of $\mathbb D_n \backslash \Gamma$. If the restriction of $\eta(a)^M$ to this component were periodic or pseudo-Anosov, this would imply that $\eta^{-1}(\mathcal C)$ can be removed from $CRS(a)$, which is not true. Hence all curves of $\eta(CRS(a))$ belong to $CRS(\eta(a))$, so the only curve which could possibly be removed from $\Gamma$ to obtain $CRS(\eta(a))$ is $\partial(\eta(\mathbb D_m))$.

Now suppose that $a_{ext}$ is trivial. Then the restriction of $\eta(a)$ (and of every power of $\eta(a)$) to the outermost component of $\mathbb D_n\backslash \eta(CRS(a))$ is also trivial. Hence the restrictions of $\eta(a)^M$ to the connected components of  $\mathbb D_n\backslash \eta(CRS(a))$ are either periodic or pseudo-Anosov, therefore $CRS(\eta(a))=\eta(CRS(a))$.

Finally, suppose that $a_{ext}$ is not trivial. Let $S$ be the outermost component of $\mathbb D_n\backslash \eta(CRS(a))$, and consider $\eta(a)^M$ restricted to $S$. This mapping class is still a braid, as $S$ is (homeomorphic to) a punctured disc. Notice that $(\eta(a)^M)_{|_S}$ is not periodic: as the last puncture of $\mathbb D_n$ does not cross any other, if it were periodic it should be trivial, but this is not the case. Also, $(\eta(a)^M)_{|_S}$ is not pseudo-Anosov, as it preserves the non-degenerate curve $\partial(\eta(\mathbb D_m))$. Hence we cannot remove $\partial(\eta(\mathbb D_m))$ from $\Gamma$ in this case, and therefore $CRS(\eta(a))=\Gamma$.

In both cases the restriction of $\eta(a)$ to the outermost component of $\mathbb D_n\backslash CRS(\eta(a))$ is trivial, that is $\eta(a)_{ext}$ is trivial. This ends the proof of Lemma~\ref{L:CRS}.
\end{proof}

Now suppose that $\eta(a)$ and $\eta(b)$ are conjugate in $B_n$, and let $\alpha\in B_n$ such that $\alpha\: \eta(a)\: \alpha^{-1} = \eta(b)$. Then we have $CRS(\eta(b))= \alpha(CRS(\eta(a)))$.

By Lemma~\ref{L:CRS}, $CRS(\eta(a))$ has a curve enclosing $m$ punctures if and only if $a_{ext}$ is trivial. Therefore $a_{ext}$ is trivial if and only if so is $b_{ext}$.

Let $CRS_{out}(\eta(a))$ be the set of outermost curves in $CRS(\eta(a))$, and define $CRS_{out}(\eta(b))$ similarly. As nested curves preserve their relative position under the homeomorphism $\alpha$, it follows that $CRS_{out}(\eta(b))= \alpha(CRS_{out}(\eta(a)))$.

Let us denote by $S_a$ the outermost connected component of $\mathbb D_n\backslash CRS(\eta(a))$ and by $S_b$ the outermost connected component of $\mathbb D_n\backslash CRS(\eta(b))$. As $\alpha$ sends $\mathcal CRS_{out}(\eta(a))$ to $\mathcal CRS_{out}(\eta(b))$, we have $\alpha(S_a)=S_b$.

 Suppose that $a_{ext}$ (and thus $b_{ext}$) is nontrivial. In this case $S_a=S_b=\mathbb D_n\backslash \eta(\mathbb D_m)$ so $\alpha$ preserves $\eta(\mathbb D_m)$, and we can clearly take a representative of $\alpha$ which fixes $\partial(\eta(\mathbb D_m))$ pointwise. Let $\alpha'$ be the restriction of this representative to $\eta(\mathbb D_m)$, and let $\alpha_0=\eta^{-1}(\alpha')\in B_m$.

Let us see that in this case $b = \alpha_0\: a \: \alpha_0^{-1}$. Indeed,
$$
 b= \eta^{-1}(\eta(b)_{|_{\eta(\mathbb D_m)}})
 = \eta^{-1}((\alpha \:\eta(a)\: \alpha^{-1})_{|_{\eta(\mathbb D_m)}}) =
$$
$$
 = \eta^{-1}(\alpha'\: \eta(a)_{|_{\eta(\mathbb D_m)}}\: (\alpha')^{-1}) = \alpha_0 \: \eta^{-1}(\eta(a)_{|_{\eta(\mathbb D_m)}})\: \alpha_0^{-1} = \alpha_0\: a\: \alpha_0^{-1}.
$$

Now suppose that $a_{ext}$ (and thus $b_{ext}$) is trivial. This implies by Lemma~\ref{L:CRS} that all curves in $CRS(\eta(a))$ and $CRS(\eta(b))$ belong to the interior of $\eta(\mathbb D_m)$. Hence $\partial(\eta(\mathbb D_m))$ belongs to the interior of both $S_a$ and $S_b$. Let $\mathcal C_1=\partial(\eta(\mathbb D_m))$ and $\mathcal C_2 = \alpha(\mathcal C_1)$. Recall that $S_b$ is isomorphic to a punctured disc, and that $\mathcal C_1$ and $\mathcal C_2$ are interior closed simple curves of $S_b$. Moreover, as the annulus determined by $\mathcal C_1$ (resp. $\mathcal C_2$) and $\partial \mathbb D_n$ contains exactly $n-m$ punctures, there is a homeomorphism of $S_b\cup CRS_{out}(\eta(b))$ which sends $\mathcal C_2$ to $\mathcal C_1$, and leaves $CRS_{out}(\eta(b))$ and $\partial \mathbb D_n$ fixed pointwise. Extending this homeomorphism by the identity to the whole $\mathbb D_n$, we obtain a homeomorphism $\beta$ of $\mathbb D_n$ with support in $S_b$, such that $\beta(\mathcal C_2)=\mathcal C_1$.

If we consider $\gamma = \beta \circ \alpha$, we notice that $\gamma(\mathcal C_1) = \beta(\alpha(\mathcal C_1)) = \beta(\mathcal C_2) = \mathcal C_1$. That is $\gamma(\partial (\eta(\mathbb D_m))) = \partial(\eta(\mathbb D_m))$. Hence $\gamma(\eta(\mathbb D_m))=\eta(\mathbb D_m)$, and we can take a representative of $\gamma$ which fixes $\partial(\eta(\mathbb D_m))$ pointwise. Moreover, $\beta$ and $\eta(b)$ have disjoint support, hence they commute. Therefore:
$$
  \gamma\: \eta(a) \gamma^{-1} = \beta \: \alpha \: \eta(a)\: \alpha^{-1} \: \beta^{-1} =
  \beta\: \eta(b) \: \beta^{-1} = \eta(b).
$$
We then have a braid $\gamma\in B_n$ which conjugates $\eta(a)$ to $\eta(b)$ and preserves $\eta(\mathbb D_m)$. Denoting $\gamma'=\gamma_{|_{\eta(\mathbb D_m)}}$, and $\gamma_0= \eta^{-1}(\gamma')\in B_m$, it follows as above that $\gamma_0\: a \: \gamma_0^{-1} = b$.

This shows that in every case, $a$ and $b$ are conjugate in $B_m$, finishing the proof of Theorem~\ref{T:main}.

\noindent {\bf Acknowledgements:} It is a pleasure for me to write this paper at the occasion of the 70th birthday of Professor Gonz\'alez-Acu\~na. The question we address in this paper was raised to me by Patrick Dehornoy, and some time later by Ivan Marin. I thank both of them, and I specially thank Ivan Marin for encouraging me to write this paper. I also thank Bert Wiest for some interesting comments, and for sending me an alternative proof of the main result using just properties of braid strands. Part of this work was done during a stay at the Centre de Recerca Matem\`atica (CRM) in Bellaterra (Barcelona, Spain), to which I thank for its hospitality.

\begin{tabular}{l}
Juan Gonz\'alez-Meneses \\  \\
Departamento de \'Algebra \\
Facultad de Matem\'aticas \\
Instituto de Matem\'aticas (IMUS)\\
Universidad de Sevilla  \\
Apdo. 1160   \\
41080 Sevilla (Spain)  \\
{\it meneses@us.es}
\end{tabular}


\begin{thebibliography}{}

\bibitem{Artin1} E.~Artin. Theorie der Z\"opfe. Abh. Math. Sem. Univ. Hamburg 4 (1925), no. 1, 47-72.

\bibitem{Artin} E.~Artin. Theory of braids.  Ann. of Math. 48  (1947), no. 2, 101-126.

\bibitem{Birman} J. S. Birman. Braids, links, and mapping class groups. Annals of Mathematics Studies, No. 82. Princeton University Press, 1974.

\bibitem{BLM} J. S. Birman, A. Lubotzky, J. McCarthy, Abelian and solvable subgroups of the mapping class groups. Duke Math. J. 50 (1983), no. 4, 1107-1120.

\bibitem{Chow} W.-L. Chow. On the algebraical braid group. Ann. of Math. (2) 49, (1948). 654-658.

\bibitem{Dyer} J. L. Dyer. The algebraic braid groups are torsion-free: An algebraic proof. Math. Z. 172 (1980), no. 2, 157-160.

\bibitem{FaN} E. Fadell, L. Neuwirth. Configuration spaces. Math. Scand. 10 1962 111-118.

\bibitem{FLP} A. Fathi, F. Laudenbach, V. Po\'enaru. Travaux de Thurston sur les surfaces. Ast\'erisque No. 66-67 (1991). SMF, Paris, 1991.

\bibitem{FoN} R. Fox, L. Neuwirth. The braid groups. Math. Scand. 10 1962 119-126.

\bibitem{GM} J.~Gonz\'alez-Meneses. On reduction curves and Garside properties of braids. Contemporary Mathematics 538 (2011), 227-244.

\bibitem{GM_survey} J. Gonz\'alez-Meneses. Basic results on braid groups. Ann. Math. Blaise Pascal 18, no. 1 (2011), 15-59.

\bibitem{PR} L. Paris, D. Rolfsen. Geometric subgroups of mapping class groups. J. reine angew. Math. 521 (2000), 47–83.

\bibitem{Thurston} W. P. Thurston. On the geometry and dynamics of diffeomorphisms of surfaces. Bull. Amer. Math. Soc. (N.S.) 19 (1988), no. 2, 417-431.

\bibitem{Wajnryb} B. Wajnryb. Artin groups and geometric monodromy. Invent. Math. 138 (1999), no. 3, 563-571.

\end{thebibliography}
\end{document}